\documentclass[a4paper,11pt]{article}
\usepackage[utf8]{inputenc}
\usepackage{a4wide}
\usepackage{algorithm}
\usepackage{algorithmic}
\usepackage{latexsym,amsfonts,amsmath,amssymb,mathrsfs,url,amsthm}
\usepackage{mathtools}
\usepackage{dsfont}
\usepackage{color,graphicx}
\usepackage{lipsum}
\usepackage{hyperref}

\usepackage{xcolor}
\newtheorem{theorem}{Theorem}[section]
\newtheorem{lemma}[theorem]{Lemma}

\newtheorem{remark}[theorem]{Remark}

\newtheorem{definition}[theorem]{Definition}

\newcommand{\E}{\mathbb{E}}
\newcommand{\NN}{\mathbb N}
\newcommand{\ZZ}{\mathbb{Z}}
\newcommand{\Ocal}{\mathcal{O}}
\newcommand{\Pcal}{\mathcal{P}}
\newcommand{\Scal}{\mathcal{S}}
\renewcommand{\P}{\mathbb{P}}

\newcommand{\bsa}{\boldsymbol{a}}
\newcommand{\bsc}{\boldsymbol{c}}
\newcommand{\bsk}{\boldsymbol{k}}
\newcommand{\bsx}{\boldsymbol{x}}
\newcommand{\bsy}{\boldsymbol{y}}
\newcommand{\bszero}{\boldsymbol{0}}

\title{Hidden low-discrepancy structures in random point sets}
\author{Kohei Suzuki\thanks{School of Engineering, The University of Tokyo, 7-3-1 Hongo, Bunkyo-ku, Tokyo 113-8656, Japan (\url{suzuki-kouhei-r66@g.ecc.u-tokyo.ac.jp}; \url{goda@frcer.t.u-tokyo.ac.jp})}, Takashi Goda\footnotemark[1]}
\date{\today}

\begin{document}

\maketitle
\sloppy

\begin{abstract}
We study the probabilistic existence of point configurations satisfying the $(0, m, d)$-net property in base $b$ within a randomly generated point set of size $N$ in the $d$-dimensional unit cube.
We first derive an upper bound on the number of geometric patterns for $(0, m, d)$-nets in base $b$.
By applying the elementary probability bounds together with this counting result, we then give scaling conditions on $N$ as a function of $m$ such that this probability converges to $1$ and $0$, respectively. 
\end{abstract}

\section{Introduction}

The dichotomy between randomized and deterministic algorithms has been a central and fundamental theme in numerical analysis \cite{No88}. In particular, in the context of high-dimensional integration, the standard Monte Carlo (MC) methods rely on independent random sampling, whereas quasi-Monte Carlo (QMC) methods rely on well-structured low-discrepancy sampling \cite{DKS13}. If the integrand, defined over the $d$-dimensional unit cube $[0,1)^d$, has a finite variance, the MC method achieves a probabilistic error of $\Ocal(N^{-1/2})$, independently of the dimension $d$. On the other hand, if the integrand has a finite total variation in the sense of Hardy and Krause, the Koksma--Hlawka inequality ensures that the QMC method based on low-discrepancy point sets achieves a deterministic error of $\Ocal(N^{-1}(\log N)^{d-1})$, which asymptotically decays much faster than the MC method \cite{KN74,Ni92}.

To bridge these approaches, randomized quasi-Monte Carlo (RQMC) methods have been extensively studied, including Cranley--Patterson rotation \cite{CP76} and Owen's scrambling \cite{Ow95}; we also refer to \cite{LL02} regarding RQMC methods. RQMC retains the structural equi-distribution properties of deterministic QMC point sets while introducing sufficient randomness to allow for statistical error estimation. Conceptually, RQMC proceeds from \emph{structure to randomness}, i.e., one starts with a highly structured point set and then applies a proper randomization.

In this paper, we explore a fundamentally different perspective, proceeding from \emph{randomness to structure}, and consider a Ramsey-theoretic question: \emph{Does a sufficiently large, purely random point set inevitably contain a highly structured subset?} To be more specific, we focus on $(0,m,d)$-nets in base $b$, a class of well-known QMC point sets \cite{Ni92,DP10}, and study whether a large random point set contains a $(0,m,d)$-net as a subset with high probability.

Our work is partly motivated by a recent work by Bansal and Jiang \cite{BJ25}, who considered downsampling $N^2$ random points using the SubgTransference algorithm. They studied its expected error for the integrand of finite ``smoothed-out'' variation, which typically decays with an order close to that expected for the QMC methods. While their work focuses on the algorithmic efficiency and error bounds of specific subsampling methods, our current work addresses the fundamental question of whether low-discrepancy structures exist within random sets.

The rest of this paper is organized as follows. Section~\ref{sec:preliminaries} introduces some necessary background and notations. Section~\ref{sec:result} presents the main result of this paper and its proof.

\section{Preliminaries}\label{sec:preliminaries}

Throughout this paper, let $\NN$ be the set of positive integers, and let $\NN_0:=\NN\cup \{0\}$. Let $b\ge 2$ be a fixed integer base. An \emph{elementary interval} in base $b$ is an axis-parallel rectangle of the form
\[ E_{\bsc,\bsa} = \prod_{j=1}^{d}\left[ \frac{a_j}{b^{c_j}}, \frac{a_j+1}{b^{c_j}}\right), \]
with $a_j, c_j\in \NN_0$, satisfying $0\le a_j<b^{c_j}$. The boldface letters such as $\bsc$ denote vectors, e.g., $\bsc=(c_1,\ldots,c_d)\in \NN_0^d$, where their lengths are clear from the context.
One can see that, given $\bsc\in \NN_0^d$, the elementary intervals $E_{\bsc,\bsa}$ partition the unit cube $[0,1)^d$, i.e.,
\[ \bigcup_{\substack{\bsa\in \NN_0^d\\ a_j<b^{c_j}}}E_{\bsc,\bsa} = [0,1)^d\qquad \text{and}\qquad E_{\bsc,\bsa_1}\cap E_{\bsc,\bsa_2}=\emptyset,\qquad \text{if $\bsa_1\ne \bsa_2$}.\]

Now we are ready to introduce a class of low-discrepancy point sets, called $(0,m,d)$-nets in base $b$.
\begin{definition}\label{def:tmd-net}
    Let $m\in \NN_0$. A point set $P$ of $b^m$ points in $[0,1)^d$ is called a \emph{$(0,m,d)$-net in base $b$} if every elementary interval $E_{\bsc,\bsa}$ with $c_1+\cdots+c_d=m$ contains exactly one point of $P$.
\end{definition}

\begin{remark}\label{rem:existence_net}
    Note that a $(0,m,d)$-net in base $b$ cannot exist if $m\ge 2$ and $d\ge b+2$, see \cite[Corollary~4.19]{DP10}. On the other hand, an explicit construction by Faure \cite{Fa82} ensures that a $(0,m,d)$-net in base $b$ exists for any prime-power base $b$, any $m\ge 2$, and any $d\le b+1$.
\end{remark}

It is well-known that the star discrepancy of a $(0,m,d)$-net in base $b$, defined as
\[ D^*(P)=\sup_{\bsy\in [0,1]^d}\left| \frac{|P\cap [\bszero,\bsy)|}{|P|}-\operatorname{vol}([\bszero,\bsy)])\right|, \]
where $[\bszero,\bsy):=[0,y_1)\times \cdots \times [0,y_d)$, is of $\Ocal(N^{-1}(\log N)^{d-1})$ with $N=b^m$, see \cite[Theorem~4.10]{Ni92}. 

Let us consider dividing each one-dimensional interval $[0,1)$ into $b^m$ equal sub-intervals of the form $[k b^{-m},(k+1)b^{-m})$ for $0\le k < b^m$. This partitions $[0,1)^d$ into $b^{md}$ sub-cubes, which can be naturally identified with the integer grid points $\ZZ_{b^m}^d := \{0, 1, \ldots, b^m-1\}^d$. In particular, a vector $\bsk=(k_1, \ldots, k_d) \in \ZZ_{b^m}^d$ corresponds to the sub-cube $\prod_{j=1}^d [k_j b^{-m}, (k_j+1)b^{-m})$. For a given point set $P \subset [0,1)^d$ with $|P|=b^m$, consider the corresponding set $\Pcal\subset \ZZ_{b^m}^d$ given by
\[ \Pcal = \left\{ (\lfloor b^m x_1\rfloor,\ldots, \lfloor b^m x_d\rfloor) \mid \bsx=(x_1,\ldots,x_d)\in P \right\}. \]
Here, the set $\Pcal$ is determined solely by the sub-cubes containing the points of $P$, regardless of their precise positions within those sub-cubes. Then, $P$ is a $(0,m,d)$-net in base $b$ if and only if all the points in $\Pcal$ are distinct and the configuration of $\Pcal$ satisfies the discrete equivalent of the equi-distribution condition imposed by the $(0,m,d)$-net property. This motivates the following definition of admissible patterns.

\begin{definition}
    A subset $\Pcal \subset \ZZ_{b^m}^d$ is called an \emph{admissible pattern} if $|\Pcal| = b^m$ and, for every $\bsc=(c_1,\ldots,c_d)\in \NN_0^d$ with $c_1+\cdots+c_d = m$ and every $\bsa=(a_1,\ldots,a_d)\in \NN_0^d$ with $0 \le a_j < b^{c_j}$, the set $\Pcal$ contains exactly one element $\bsk=(k_1, \ldots, k_d)$ satisfying
    \[
        a_j b^{m-c_j} \le k_j < (a_j+1) b^{m-c_j} \quad \text{for all } j=1, \ldots, d.
    \]
    We denote the cardinality of the set of all such patterns by $a_{b,d}(m)$.
\end{definition}

\section{Results}\label{sec:result}

We now state the main result of this paper. Determining the convergence behavior in the intermediate regime between the two growth conditions on $N$, given in the theorem, is left open for future work.
\begin{theorem}\label{thm:main}
    For $d\in \NN$, let $\Scal_N$ be a set of $N$ independently and uniformly distributed random points in $[0, 1)^d$. Let $b,m\ge 2$ such that $b$ is a prime-power with $d\le b+1$. Let $C_{b,d}(N,m)$ denote the event that $\Scal_N$ contains a $(0,m,d)$-net in base $b$. Considering $N:=N(m)$ as a function of $m$, the following holds:
    \begin{enumerate}
        \item If $N(m)\ge (1+\epsilon)\, b^{md}m\log b$ for some constant $\epsilon>0$, then 
        \[ \lim_{m\to \infty}\P[C_{b,d}(N(m),m)] = 1.\]
        \item If $N(m)<(1-\epsilon)\,b^{md}/(b!)^{m(d-1)/b}< (1-\epsilon)\,b^me^{(1-1/b)m(d-1)}$ for some constant $\epsilon\in (0,1)$, then
        \[ \lim_{m\to \infty}\P[C_{b,d}(N(m),m)] = 0.\]
    \end{enumerate}
\end{theorem}

To prove this result, we first give an estimate on $a_{b,d}(m)$. Although Remark~\ref{rem:existence_net} implies that $a_{b,d}(m)=0$ if $m\ge 2$ and $d\ge b+2$, here we provide an upper bound on $a_{b,d}(m)$, which is applicable to any $d,m\in \NN$ and $b\ge 2$.

\begin{lemma}\label{lem:counting}
    For any $d,m\in \NN$ and $b\ge 2$, we have
    \[ a_{b,d}(m)\le (b!)^{mb^{m-1}(d-1)}.\]
\end{lemma}

\begin{proof}
    The case $d=1$ is trivial since there is only one admissible pattern, which is consistent with the bound (with exponent $0$). For $d=2$, we rely on the recursive construction of $(0,m,2)$-nets proposed by Leobacher, Pillichshammer, and Schell \cite{LPS14}. It suffices to prove that their Algorithm~2 establishes a bijection between the set of $(0,m,2)$-nets and the set of tuples consisting of $b$ $(0,m-1,2)$-nets augmented with structural permutations. Specifically, any $(0,m,2)$-net can be uniquely decomposed into $b$ vertical strips 
    \[ I_k = \left[ \frac{k}{b},\frac{k+1}{b}\right)\times [0,1),\]
    for $0\le k < b$. Each strip, upon appropriate scaling, forms a $(0,m-1,2)$-net. Conversely, to reconstruct a $(0,m,2)$-net from $b$ such sub-nets, one must determine the $m$-th $b$-adic digit of the $y$-coordinates. The domain is divided into $b^{m-1}$ horizontal rows of height $b^{-(m-1)}$. In each row, there are exactly $b$ points (one from each vertical strip). To satisfy the $(0,m,2)$-net property, these $b$ points must have distinct $y$-digits at the $m$-th level, which corresponds to choosing a permutation $\sigma\in S_b$. Since there are $b^{m-1}$ rows and the choices of permutations are independent, we have
    \[ a_{b,2}(m)=(a_{b,2}(m-1))^b \cdot (b!)^{b^{m-1}}. \]
    This recursion, together with the trivial result for the base case $a_{b,2}(0)=1$, gives
    \[ a_{b,2}(m)=(b!)^{mb^{m-1}}. \]

    For general $d\ge 2$, we use the trivial fact that any two-dimensional projection of a $(0,m,d)$-net is a $(0,m,2)$-net. Given that, for any $2\le j\le d$, the two-dimensional projection onto the first and the $j$-th coordinates can have $a_{b,2}(m)$ patterns, we simply have
    \[ a_{b,d}(m)\le (a_{b,2}(m))^{d-1}=(b!)^{mb^{m-1}(d-1)}.\]
    This completes the proof.
\end{proof}

We are now ready to prove Theorem~\ref{thm:main}.
\begin{proof}[Proof of Theorem~\ref{thm:main}]
    We prove the case $d\ge 2$ first; the case $d=1$ shall be proven later.
    Throughout this proof, let $A:=a_{b,d}(m)$ denote the total number of admissible patterns. For each integer $k$ with $1\le k\le A$, let $X_k$ be an indicator random variable that equals $1$ if every sub-cube constituting the $k$-th pattern contains at least one point from $\Scal_N$, and $0$ otherwise. Note that if $X_k=1$, since the sub-cubes in a pattern are disjoint, we can select exactly one point from each non-empty sub-cube to form a subset of $\Scal_N$ that constitutes a $(0,m,d)$-net. We define the random variable $X$, representing the number of realized patterns, by
    \[ X := \sum_{k=1}^{A} X_k. \]    
    The condition $X>0$ signifies that at least one admissible pattern is contained in $\Scal_N$. This implies that the probability we seek in this theorem is $\P[C_{b,d}(N,m)]=\P[X>0]$.
    
    Since there exists at least one admissible pattern (see Remark~\ref{rem:existence_net}), i.e., $A\ge 1$, we have the following chain of inclusion of events
    \[ \{X_1=1\} \subset \{X>0\} = \bigcup_{k=1}^{A}\{X_k=1\}. \]
    Then, a union bound gives us
    \begin{equation}\label{eq:sandwich}
    p_N(b^m)=\P(X_1=1)\le \P(X > 0) \le \sum_{k=1}^{A}\P(X_k=1)=A\cdot p_N(b^m).
    \end{equation}
    Here $p_N(b^m)$ denotes the probability that, for an arbitrarily fixed set of $b^m$ distinct sub-cubes, every sub-cube contains at least one point from the $N$ random points in $\Scal_N$.
    
    Let us prove the first claim of the theorem. To do this, consider the event $X_1=1$. Let $E_j$ be the event that the $j$-th sub-cube in the specific pattern, characterized by $X_1$, is empty, for $j=1, \ldots, b^m$. 
    Obviously, the probability that a single sub-cube is empty after $N$ draws is
    \[ \P(E_j) = \left( 1 - \frac{1}{b^{md}} \right)^N \le \exp\left( - \frac{N}{b^{md}} \right). \]
    By the union bound, we have
    \begin{align*}
        1 - p_N(b^m) = \P\left( \bigcup_{j=1}^{b^m} E_j \right) \le \sum_{j=1}^{b^m} \P(E_j) \le b^m \exp\left( - \frac{N}{b^{md}} \right). 
    \end{align*}
    Taking the logarithm of the right-hand side, we get
    \[ \log \left( b^m \exp\left( - \frac{N}{b^{md}} \right) \right) = m \log b - \frac{N}{b^{md}}. \]
    
    If we choose $N$ to be an integer greater than or equal to $(1+\epsilon)\, b^{md} m \log b$ for some $\epsilon > 0$, then the exponent tends to $-\infty$ as $m \to \infty$:
    \[ m \log b - (1+\epsilon) m \log b = -\epsilon m \log b \to -\infty. \]
    Consequently, $1 - p_N(b^m) \to 0$, or equivalently $p_N(b^m) \to 1$. It follows from the lower bound in \eqref{eq:sandwich} that $\lim_{m\to \infty}\P(X > 0) = 1$, which proves the claim.

    Let us move on to the second claim of the theorem. Note that while the events of filling distinct sub-cubes are not independent, they are negatively associated \cite{JP83}. Thus, the joint probability is bounded by the product of marginal probabilities:
    \[ p_N(b^m) \le \left( 1 - \left( 1 - \frac{1}{b^{md}} \right)^N \right)^{b^m} \le \left( N \cdot \frac{1}{b^{md}} \right)^{b^m}, \]
    where we used the inequality $1-(1-x)^N\le Nx$ for $x\in [0,1]$.
    Together with Lemma~\ref{lem:counting}, we have
    \[ A\cdot p_N(b^m)\le (b!)^{mb^{m-1}(d-1)}\left( N \cdot \frac{1}{b^{md}} \right)^{b^m}. \]

    Now let us choose $N$ to be an integer less than $(1-\epsilon)\,b^{md}/(b!)^{m(d-1)/b}$ for some constant $\epsilon\in (0,1)$. Then, since $b!> e(b/e)^b$ for any integer $b\ge 2$, we see that $N$ is further bounded above by $(1-\epsilon)\,b^me^{(1-1/b)m(d-1)}$. It holds that
    \begin{align*}
        A\cdot p_N(b^m)< (b!)^{mb^{m-1}(d-1)}\left( (1-\epsilon)\frac{b^{md}}{(b!)^{m(d-1)/b}} \cdot \frac{1}{b^{md}} \right)^{b^m}=(1-\epsilon)^{b^m}.
    \end{align*}
    Since $\epsilon\in (0,1)$ and $b\ge 2$, we have $(1-\epsilon)^{b^m}\to 0$ as $m\to \infty$. It follows from the upper bound in \eqref{eq:sandwich} that $\lim_{m\to \infty}\P(X > 0) = 0$, which proves the claim.

    Finally, we prove the case $d=1$. In this case, we always have $A=1$ independently of $b$ and $m$. This leads to the equality $\P[C_{b,1}(N,m)]=p_N(b^m)$. Thus, the lower and upper estimates for $p_N(b^m)$ provided above directly apply.
\end{proof}

\begin{remark}
    The following observation has been pointed out to the authors by Michael Gnewuch (private communication, 2025).
    Suppose we fix a specific admissible pattern consisting of $b^m$ sub-cubes. The number of random samples $N$ required to fill all these sub-cubes follows a distribution similar to the coupon collector's problem, where we aim to collect $b^m$ specific coupons out of $b^{md}$ possibilities. The probability of hitting a new sub-cube in the target pattern, given that $i-1$ have already been filled, is $q_i = (b^m - (i-1))/b^{md}$. The expected number of samples is thus
    \[ \E[N] = \sum_{i=1}^{b^m} \frac{1}{q_i} = b^{md} \sum_{j=1}^{b^m} \frac{1}{j} = b^{md} H_{b^m} \approx b^{md} m \log b, \]
    where $H_{b^m}$ denotes the $b^m$-th harmonic number. This implies that our condition to ensure $\lim_{m\to \infty}\P[C_{b,d}(N(m),m)] = 1$ is essentially the expected number of random points necessary to find a specific admissible pattern, up to a constant factor $1+\epsilon$. 
\end{remark}

\section*{Acknowledgments}
The authors would like to thank Michael Gnewuch for his valuable suggestions. We are also grateful to the anonymous reviewers for their careful reading and constructive comments, which helped improve the manuscript.

\section*{Funding}
The work of the second author (T.G.) is supported by JSPS KAKENHI Grant Number 23K03210.

\bibliographystyle{amsplain}
\bibliography{references}

\end{document}